\documentclass[12pt,a4paper]{article}
\usepackage[utf8x]{inputenc}
\usepackage{amsmath}
\usepackage{amssymb}
\usepackage{amsthm}
\usepackage{dsfont}
\usepackage{graphicx}
\usepackage{mathtools}

\numberwithin{equation}{section}

\newcommand{\Z}{{\mathbb Z}}

\newcommand{\R}{{\mathbb R}}

\renewcommand{\phi}{\varphi}

\newcommand{\IE}{\mathbb{E}}
\newcommand{\IP}{\mathbb{P}}

\newcommand{\Var}{\mathop{\mathrm{Var}}}
\newcommand{\dist}{\mathop{\mathrm{dist}}}
\newcommand{\diam}{\mathop{\mathrm{diam}}}

\allowdisplaybreaks

\newcommand{\I}[1]{{\mathds{1}}_{#1}}

\newtheorem{theo}{Theorem}[section]

\newtheorem{prop}[theo]{Proposition}

\newtheorem{cor}[theo]{Corollary}
\newtheorem{rmk}[theo]{Remark}

\title{On decoupling inequalities and percolation of
excursion sets of the Gaussian free field}
\author{Serguei~Popov$^{1}$ \and Bal\'azs R\'ath$^{2}$}

\begin{document}

\maketitle

{\footnotesize 
\noindent $^{~1}$Department of Statistics, Institute of Mathematics,
 Statistics and Scientific Computation, University of Campinas --
UNICAMP, rua S\'ergio Buarque de Holanda 651,
13083--859, Campinas SP, Brazil\\
\noindent e-mail: \texttt{popov@ime.unicamp.br}

\noindent $^{~2}$Budapest University of Technology, Institute of Mathematics,
MTA-BME Stochastics Research Group, 1 Egry J\'ozsef u., 1111 Budapest, Hungary.\\
\noindent e-mail: \texttt{rathb@math.bme.hu}
}

\begin{abstract}
We prove decoupling inequalities for the Gaussian free field on~$\Z^d$, $d\geq 3$.
As an application, we obtain  exponential decay 
(with logarithmic correction
for $d=3$) of the connectivity function of excursion
sets for large values of the threshold.
\\[.2cm]\textbf{Keywords:} percolation, Gaussian free field,
simple random walk, random interlacements
\\[.2cm]\textbf{AMS 2000 subject classifications:}
Primary 60K35, 82B43. 
\end{abstract}

\section{Correlation and decoupling inequalities}
\label{s_decoupling}
Let us denote by $\|x\|$ the Euclidean norm of $x\in\Z^d$.
The Gaussian free field (GFF) on $\Z^d$, $d\geq 3$,
 is a centered Gaussian 
field $\phi = (\phi_x)_{x\in\Z^d}\in\R^{\Z^d}$
under  the  probability measure $\mathbb P$ with covariance 
$\IE (\phi_x\phi_y) = g(x,y)$ 
for all $x,y\in\Z^d$,
where $g(\cdot,\cdot)$ 
denotes the Green function of 
the simple random walk on $\Z^d$. 
The random field~$\phi$ exhibits long-range correlations, 
since
\begin{equation}\label{lawler_green}
c_g \Vert x-y \Vert^{2-d} \leq  g(x,y) \leq C_g \Vert x-y \Vert^{2-d}, \;\; x \neq y 
\end{equation}
for some $0<c_g(d) \leq C_g(d) <\infty$,  see \cite[Theorem 1.5.4]{lawler}.

The goal of this section is to quantify the dependence between configurations
of the GFF, supported on disjoint 
(and, usually, distant)
sets.  Decoupling inequalities of this kind
 are useful tools in the study of percolation of 
the excursion sets of the GFF 
 (see~\cite[Proposition 2.2]{RodSz} 
and~\cite[Lemma 2.6]{DreRathSap}) and the vacant set of 
 random interlacements (see~\cite[Theorem 2.6]{Sznitman:Decoupling}). 
When starting to work on the subject of the present paper,  
our initial aim was to find the GFF counterpart of the decoupling inequalities proved for 
 random interlacements in~\cite[Theorem 1.1]{PopovTeixeira} (see 
 Remark~\ref{remark_similarities} for further discussion), but our
 Theorem~\ref{t_cond} actually gives a stronger, 
conditional form of decoupling inequalities for events defined in 
terms of the GFF. In Section~\ref{s_connect} we  give an application
 of the decoupling inequalities stated in this section.
 
\medskip 
 If $K, K' \subseteq \Z^d$, we define 
 \[
\dist(K,K')=\min_{x\in K,\, y\in K'}\|x-y\|,
\qquad  \diam(K)=\sup_{x,y \in K} \|x-y\|.
 \] 

 With a slight abuse of notation, for $K\subset\Z^d$ we
write $\phi_K:=(\phi_x)_{x\in K}$ for the GFF
restricted to~$K$. 
We say that a (measurable) function~$f: \R^{\Z^d} \rightarrow \R$
is supported on~$K$ if for any $\eta,\eta'\in\R^{\Z^d}$ 
such that $\eta_K=\eta'_K$ it holds that $f(\eta)=f(\eta')$.
Let us also define 
\[
 \mathrm{Cov}\big(f_1(\varphi), f_2(\varphi)\big)
=\IE\big(f_1(\varphi) f_2(\varphi)\big)
 - \IE f_1(\varphi) \IE f_2(\varphi).
\]

If~$K$ is a finite subset of $\Z^d$, we define the capacity of~$K$ by the formula
\[
 \mathrm{cap}(K) = \sum_{x \in K} P_x[ \widetilde{H}_K = +\infty], 
\]
where $P_x$ is the law of simple random walk~$X$ on $\Z^d$ 
started from $x \in \Z^d$ and 
$\widetilde{H}_K = \min\{n\geq 1 : X_n\in K\}$ 
is the hitting time of the set~$K$.

First, we formulate the following result about correlations
of functions supported on disjoint finite sets,
which is in the spirit of the basic correlation inequality
for random interlacements, see \cite[(2.15)]{SznitmanAM}.
\begin{prop} 
\label{p_simple}
Let $d \geq 3$. 
There exist constants $0<c_d \leq C_d < +\infty$ 
that depend only on $d$ such that if
$K_1, K_2$ are disjoint finite subsets of~$\Z^d$ and 
\begin{equation}\label{K_1_K_1_far}
 \dist(K_1,K_2) \geq \max \{ \diam(K_1), \diam(K_2) \}, 
\end{equation}
then
\begin{equation}\label{our_corr_ineq}
 c_d 
\frac{\big(\mathrm{cap}(K_1)\mathrm{cap}(K_2)\big)^{1/2}}{\dist(K_1,K_2)^{d-2}} \leq
\sup_{f_1, f_2} \mathrm{Cov}\big(f_1(\varphi), f_2(\varphi)\big) 
\leq C_d 
\frac{ \big(\mathrm{cap}(K_1) \mathrm{cap}(K_2)\big)^{1/2}}{ \dist(K_1,K_2)^{d-2} },
\end{equation}
where the supremum is taken over $[0,1]$-valued functions $f_1, f_2: \R^{\Z^d} \rightarrow [0,1] $, 
and where $f_1$ is supported on $K_1$ and $f_2$ is supported on $K_2$. 
\end{prop}
Let us remark that the assumption \eqref{K_1_K_1_far} is
 only used in the proof the lower bound of \eqref{our_corr_ineq}.

While the above result does indeed give the right order
of decay of correlations, it is not always the right tool one looks for.
The reason is that the covariance decreases polynomially
in distance, which makes renormalization arguments more difficult.
One can circumvent this problem by using the method of \emph{sprinkling}, which 
has been effectively applied to produce powerful decoupling inequalities
for the excursion sets of the GFF 
 (see~\cite[Proposition 2.2]{RodSz} 
and~\cite[Lemma 2.6]{DreRathSap}) and the vacant set of 
 random interlacements (see~\cite[Section 2]{Sznitman:Decoupling}).

To explain this approach, we need more definitions.
Write $\eta\leq \eta'$ if $\eta_x\leq \eta'_x$ for all $x\in\Z^d$.
A function~$f: \R^{\Z^d} \rightarrow \R$ is called \emph{increasing}
if $\eta\leq \eta'$ implies $f(\eta)\leq f(\eta')$,
and \emph{decreasing} if $(-f)$ is increasing. 
For $\eta\in \R^{\Z^d}$ and $a\in\R$ we use the shorthand
$\eta+a$ for the configuration defined by 
$(\eta+a)_x=\eta_x+a$, $x\in\Z^d$. 
 
 Let us fix two disjoint sets $K_1,K_2\subset \Z^d$, such that $K_1$ is finite.
 It is known (see~\cite[Proposition 2.3]{Sznitman:Topics} 
and~\cite[Lemma 1.2]{RodSz}) 
that $\mathbb{P}$-a.s.\ there exists a  decomposition 
 \begin{equation}\label{decomposition}
 \phi={\tilde \phi}+h
 \end{equation}
  into a
sum of independent Gaussian fields, 
 where~${\tilde \phi}$
is a centered field such that ${\tilde \phi}_{K_1}\equiv 0$, 
and 
\begin{equation}\label{def_eq_h}
  h_x = \sum_{y \in K_1} 
P_x[ H_{ K_1} < \infty, \,  X_{H_{K_1}} =y] \cdot \phi_y, \qquad    \text{$\mathbb{P}$-a.s.},
\end{equation}
 where   $H_{K_1} = \min\{n\geq 0 : X_n \in K_1\}$ 
is the entrance time of the random walk to the set~$K_1$.
Note that~$h$ is measurable with 
respect to the sigma-algebra generated by~$\phi_{K_1}$ and that one has $h_{K_1} \equiv \phi_{K_1}$.

For $\delta>0$, define the $\sigma(\phi_{K_1})$-measurable event
\begin{equation}
\label{def_G}
 G_\delta = \Big\{\sup_{x\in K_2}|h_x|\leq \frac{\delta}{2}\Big\}.
\end{equation}

Our main result is about
 the \emph{conditional} decoupling
inequalities:
\begin{theo}
\label{t_cond}
Assume that $f_2:\R^{\Z^d}\rightarrow [0,1]$ is increasing and
supported on~$K_2$. For all $\delta>0 $ it
$\mathbb{P}$-a.s.\ holds that 
\begin{equation}
\label{cond_incr}
 \big( \IE(f_2(\phi-\delta))-\IP[G_\delta^c]\big)\I{G_\delta} \leq 
 \IE\big(f_2(\phi)\mid \phi_{K_1}\big)\I{G_\delta}  \leq
 \big(\IE(f_2(\phi+\delta))+\IP[G_\delta^c]\big)\I{G_\delta}.
\end{equation}
\end{theo}

An (almost) immediate consequence of this result is the following
\begin{cor}
\label{t_decouple}
Assume that $f_2:\R^{\Z^d}\rightarrow [0,1]$ is increasing and
supported on~$K_2$, and $f_1:\R^{\Z^d}\rightarrow [0,1]$ is 
\emph{any} function supported on~$K_1$. Then for all $\delta>0$
we have
\begin{equation}
\label{uncond_incr}
\IE f_1(\phi)\IE f_2(\phi-\delta)-2\IP[G_\delta^c]\leq 
\IE f_1(\phi)f_2(\phi) \leq \IE f_1(\phi)
\IE f_2(\phi+\delta)+2\IP[G_\delta^c].
\end{equation}
\end{cor}
It is straightforward to see that the corresponding
results (with obvious changes)
 also hold if $f_2$ is a \emph{decreasing} function.

We refer to the quantity~$\delta$ in~\eqref{cond_incr}--\eqref{uncond_incr}
as the amount of sprinkling. In Proposition \ref{p_ocG} below
 we shall see that the term $\IP[G_\delta^c]$
 decreases quite fast as $\dist(K_1,K_2)$ increases;
so one can decrease the correlation term of Proposition \ref{p_simple} at the cost of
``changing the level'' of the field for the monotone function $f_2$.

\medskip

The next proposition tells us how to choose the amount of sprinkling $\delta$
if we want a useful
 upper bound for~$\IP[G_\delta^c]$.
Let us denote by~$|K|$ the cardinality of~$K\subset\Z^d$. 
For any $K \subseteq \Z^d$ we define 
\begin{align*}
 K^{(\geq s)}&=\{y\in\Z^d: \dist(y,K) \geq s\}
 \\
 K^{(= s)}&=\{y\in\Z^d: s\leq \dist(y,K) <s+1\}.
\end{align*}
Note that a nearest-neighbor walk from $K$ to $K^{(\geq s)}$ must pass through~$K^{(= s)}$.

Having fixed the disjoint subsets $K_1$ and $K_2$ of $\Z^d$, let us define $s=\dist(K_1,K_2)>0$ and
the auxiliary sets $H_1,H_2 \subseteq \Z^d$ in the following way:
\begin{itemize}
\item if $\diam(K_1)\leq \diam(K_2)$,  define
 $H_{2}=K_1^{(\geq s)}$ and $H_1=H_{2}^{(\geq s)}$,
\item if $\diam(K_1) > \diam(K_2)$,  define
 $H_{1}=K_2^{(\geq s)}$ and $H_2=H_{1}^{(\geq s)}$.
\end{itemize}
 With these definitions we have
 \begin{equation}\label{auxialiary_set_properties}
  K_i \subseteq H_i, \quad H_{3-i}=H_i^{(\geq s)}, \quad |H_i^{(= s)}|<\infty, \quad i\in\{1,2\}. 
  \end{equation}

Let us define
\[
 g_s = \sup_{y:\, \|y\|\geq s} g(0,y).
\]

\begin{prop}
\label{p_ocG}
Denote by~$s=\dist(K_1,K_2)>0$.
Then
\begin{equation}
\label{eq_ocG}
\IP[G_\delta^c] \leq 
2 |H_1^{(=s)}| \exp\Big(-\frac{\delta^2}{8g_s} \Big).
\end{equation}
\end{prop}

\begin{rmk}\label{remark_similarities}
Let us compare the inequality obtained from the combination of \eqref{uncond_incr} and \eqref{eq_ocG} with the main decoupling result (Theorem~2.1) of~\cite{PopovTeixeira}.
 Observing that $g_s=\mathcal{O}(s^{2-d})$ by \eqref{lawler_green}, one can
note the similarity of the expression in the exponent in the right-hand side
of~\eqref{eq_ocG} with that in the error term
in \cite[Theorem~2.1]{PopovTeixeira}. Also, let~$r$ be the \emph{minimum} of the Euclidean diameters 
 of $K_1$ and $K_2$. Then, one can (very crudely) bound $|H_1^{(=s)}|$
by $\text{const}\times (r+s)^d$ and again observe the similarity with
the error term in \cite[Theorem~2.1]{PopovTeixeira}.
\end{rmk}

Before we prove the results stated above in Section~\ref{s_proofs}, 
we give an application of our decoupling inequality in Section \ref{s_connect}.

\section{Connectivity decay for percolation of excursion sets}
\label{s_connect}

For any $h\in\R$, we define 
 the \textit{excursion set} above level~$h$ as 
\begin{equation*}
E^{\geq h}_\phi = \{x\in\Z^d: \phi_x \geq h\} .
\end{equation*}
We view $E^{\geq h}_\phi$ as a random 
subgraph of~$\Z^d$, and, 
naturally, one may be interested 
 in studying its percolation properties.
Let us write $\{x \xleftrightarrow{\geq h} y\}$ for the 
event when $x,y\in\Z^d$ are connected in $E^{\geq h}_\phi$.
 As an application of Corollary~\ref{t_decouple}, in this section
we establish a result on the decay of the connectivity function of the 
excursion set $E^{\geq h}_\phi$. 

In~\cite{BLM} (for $d=3$) and in~\cite{RodSz} (for all $d \geq 3$) 
it was shown that there exists $h_{*}=h_{*}(d)\in[0,+\infty)$
such that $E^{\geq h}_\phi$ percolates for $h<h_{*}$ and does not 
percolate for $h>h_{*}$; moreover, it was also proved that 
$h_{*}(d)>0$ for all sufficiently large~$d$. 
Further
developments regarding percolation of $E^{\geq h}_\phi$ and its connection to
interlacement percolation (based on the isomorphism theorem \cite{Sznitman_isomorph})
  can be found in~\cite{Lupu,Rod}.

In~\cite[(0.6)]{RodSz} the authors define another critical
parameter $h_{**}\geq h_{*}$ as the threshold above which one 
has at least polynomial decay of the probabilities
of certain crossing events:
\[ h_{**}(d)= \inf \{ h \in \R \; ; \; \text{for some } \alpha>0,
\lim_{L \to \infty} L^{\alpha} \mathbb{P}[ B(0,L) \stackrel{\geq h}{\longleftrightarrow} S(0,2L) ] =0 \},
\]
where the event $\{ B(0,L) \stackrel{\geq h}{\longleftrightarrow} S(0,2L) \}$ refers to the existence of 
a nearest-neighbour path in $E^{\geq h}_\phi$ connecting $B(0,L)$, the ball of radius $L$ around $0$ in the
$\ell^{\infty}$-norm, to $S(0,2L)$, the $\ell^{\infty}$-sphere of radius $2L$ around $0$.

In \cite[Theorem 2.6]{RodSz} they prove that $h_{**}(d)<\infty$ for all $d\geq 3$, and
also give a stretched exponential upper
bound for the connectivity function $\IP[0 \xleftrightarrow{\geq h} x]$
as $\|x\|\to\infty$ for $h>h_{**}$.
In the next theorem, we further weaken the definition of~$h_{**}$ 
and, more importantly, we improve on the stretched exponential
 bound for values of~$h$ above~$h_{**}$.

\begin{theo}
 \label{t:connect}
 For $d \geq 4$, given $h > h_{**}(d)$, there exist positive constants 
 $\gamma_1 = \gamma_1(d,h)$ and $\gamma_2 = \gamma_2(d,h)$ such that
 \begin{equation}
\label{e:connect_d4}
  \IP[0 \xleftrightarrow{\geq h} x] \leq
\gamma_1 \exp \{ -\gamma_2 \|x\|\},
\text{ for every $x \in \Z^d$.}
 \end{equation}
 If $d = 3$ and $h > h_{**}(3)$, then for any $b>1$ there exist
$\gamma'_1 = \gamma'_1(h,b)$ and
$\gamma'_2 = \gamma'_2(h,b)$
such that
 \begin{equation}
\label{e:connect_d3}
  \IP[0 \xleftrightarrow{\geq h} x] \leq
\gamma'_1 \exp \Big\{ -\gamma'_2
\frac{\|x\|}{\log^{3b}\|x\|} \Big\}, \text{ for every $x \in \Z^d$.}
 \end{equation}

 Moreover, we show that the quantity~$h_{**}$ can be defined as
 \begin{equation}
 \label{e:ustarstar3}
  h_{**} = \inf \Big\{ h > 0;\; \liminf_{L \to \infty}
\IP \big[ [0,L]^d \xleftrightarrow{\geq h}
\partial [-L,2L]^d \big] < \frac{7}{2d\cdot 21^d} \Big\}.
 \end{equation}
\end{theo}

\begin{proof}[Proof of Theorem~\ref{t:connect}]
The reader may have noticed that the above result 
is a copy of the statement of \cite[Theorem~3.1]{PopovTeixeira}, 
with obvious notational changes.
Indeed, as observed in \cite[Remark~3.4]{PopovTeixeira}, 
the proof of that theorem can be adapted to \emph{any}
percolation model which satisfies certain monotonicity and 
decoupling properties, that the excursion sets of the
GFF do possess. 

More specifically, let us denote by $\IP_h$ the law 
of the excursion set $E^{\geq h}_\phi$.
 Then~\eqref{uncond_incr} and~\eqref{eq_ocG}
imply that for any increasing events $A_1,A_2$
 that depend on disjoint boxes of
size~$r$ within distance at least~$s$ from each other,
we have
\begin{equation}
\label{strongdecoupling}
\IP_h[A_1\cap A_2] \leq \IP_{h-\delta}[A_1]
 \IP_{h-\delta}[A_2]
    + C (r+s)^d 
 \exp(-C'\delta^{2} s^{d-2}).
\end{equation}
This decoupling inequality is a special case of the one in~\cite[Remark 3.4]{PopovTeixeira}, 
the proof of this result is also practically a copy
of the proof of~\cite[Theorem~3.1]{PopovTeixeira}.
\end{proof}

It is important to observe that~\eqref{strongdecoupling} may be seen as
a partial replacement of the BK inequality, which is very useful for
proving exponential decay of crossing probabilities in the subcritical phase of classical (Bernoulli) percolation,
see e.g.\ Section~5.2 of~\cite{Grimmett}. While the ``pure''
BK inequality generally does not hold in the dependent percolation models
we mentioned here, inequalities similar to~\eqref{strongdecoupling}
are still very useful, even though they usually involve sprinkling
and the additive ``error'' term. It turns out that, for renormalization
arguments, this additive term is pivotal, in the 
sense that the smaller it is, the better results one obtains for the 
decay of the probabilities
of crossing events. In fact one can even achieve exponential decay (if $d \geq 4$)
 in polynomially correlated percolation models without the BK inequality.

Another important observation is that the question whether $h_*=h_{**}$
is still open (as well as the corresponding question for random interlacements).
In our opinion, the conditional
decoupling result of our Theorem \ref{t_cond}
might help in proving it; in fact, it did help in the proof of the fact that
$h_*(d)/h_{**}(d) \to 1$ as $d\to\infty$, see~\cite{DreRod}.

\section{Proofs of the decoupling results}
\label{s_proofs}
We start by deducing the correlation bounds with no sprinkling from
the corresponding general results of \cite[Chapter 10]{Janson}.
\begin{proof}[Proof of Proposition~\ref{p_simple}]
First note that a function $f: \R^{\Z^d} \rightarrow \R$
is supported on~$K$ if and only if $f$ is $\sigma(\phi_{K})$-measurable.

Denote by $\mathbf{G}$ the Gaussian Hilbert space that arises as the closure of vector space of linear combinations of
 $\varphi_x, \; x \in \Z^d$ under the norm $\Vert X \Vert= \sqrt{ \langle X, X \rangle }$ given by the
  inner product $\langle X,Y \rangle =\mathrm{Cov}(X,Y)$.

 Denote by $\mathbf{H}$ and $\mathbf{K}$ the subspaces of $\mathbf{G}$ spanned by
  linear combinations of $\varphi_x, \; x \in K_1$
 and  $\varphi_x, \; x \in K_2$, respectively.
Let us introduce the sigma-algebras $\mathcal{F}=\sigma(\mathbf{H})=\sigma(\varphi_{K_1})$ and 
$\mathcal{G}=\sigma(\mathbf{K})=\sigma(\varphi_{K_2})$.

  Recall from \cite[Definitions 10.5, 10.6]{Janson} the notion of the strong mixing coefficient
 $\alpha(\cdot,\cdot)$ and the maximal correlation coefficient $\rho(\cdot,\cdot)$:
  \begin{align*}
\alpha(\mathbf{H},\mathbf{K}) = \alpha(\mathcal{F},\mathcal{G}) &= \sup_{A \in \mathcal{F}, B \in \mathcal{G}} 
 | \mathbb{P}[A \cap B] - \mathbb{P}[A] \mathbb{P}[B] |, \\
\rho(\mathbf{H},\mathbf{K}) = \rho(\mathcal{F},\mathcal{G})
 &= \sup_{X \in L^2(\mathcal{F}), \, Y \in L^2(\mathcal{G})} 
\frac{ \mathrm{Cov}(X,Y)}{ \sqrt{ \mathrm{Var}(X) \mathrm{Var}(Y) } }.
 \end{align*}
With this notation we have 
\begin{equation}\label{squeeze}
 \alpha(\mathbf{H},\mathbf{K})
  \stackrel{(*)}{\leq}
\sup_{f_1, f_2} \mathrm{Cov}\big(f_1(\varphi), f_2(\varphi)\big) 
\stackrel{(**)}{\leq} \rho(\mathbf{H},\mathbf{K}),
\end{equation}
where the supremum is taken over $[0,1]$-valued functions $f_1, f_2: \R^{\Z^d} \rightarrow [0,1] $, and
where $f_1$ is supported on $K_1$ and $f_2$ is supported on $K_2$. Indeed, $(*)$ follows if we choose
$f_1=\I{A}$ and $f_2= \I{B}$ (or $f_2= \I{B^c}$ if $\mathrm{Cov}\big( \I{A}, \I{B} \big)<0$) and $(**)$ follows because 
$\mathrm{Var}(f_1),\mathrm{Var}(f_2) \leq 1$.

Denote by $P_{\mathbf{K}}: \mathbf{G} \to \mathbf{K}$ the orthogonal projection to the subspace $\mathbf{K}$ and by 
$P_{\mathbf{HK}}: \mathbf{H} \to \mathbf{K}$ the restriction of $P_{\mathbf{K}}$ to $\mathbf{H}$. 
Similarly, let $P_{\mathbf{H}}$ denote the projection to $\mathbf{H}$ and $P_{\mathbf{KH}}$ its restriction to $\mathbf{K}$.
Denote by $\Vert  \cdot  \Vert$ the operator norm of linear operators on (subspaces of) $\mathbf{G}$.

 \cite[Theorem 10.11]{Janson} states that
 $\rho(\mathbf{H},\mathbf{K})=\Vert P_{\mathbf{HK}} \Vert$ and  \cite[Remark 10.1(ii), Theorem 10.13]{Janson} imply that
 $\frac{1}{2\pi}\Vert P_{\mathbf{HK}} \Vert \leq \alpha(\mathbf{H},\mathbf{K})$. 
 Combining these results with~\eqref{squeeze} we see that in order to prove
 Proposition~\ref{p_simple}  we only need to  show that there exist constants $0<c'_d \leq C'_d < +\infty$ 
that depend only on $d$ such that if $K_1, K_2$ satisfy \eqref{K_1_K_1_far} then we have 
\begin{equation}\label{rewritten_cov_norm}
c'_d 
\frac{\big(\mathrm{cap}(K_1)\mathrm{cap}(K_2)\big)^{1/2}}{\dist(K_1,K_2)^{d-2}} \leq
\Vert P_{\mathbf{HK}} \Vert
\leq C'_d 
\frac{ \big(\mathrm{cap}(K_1) \mathrm{cap}(K_2)\big)^{1/2}}{ \dist(K_1,K_2)^{d-2} }.
\end{equation}
First note  that the adjoint of $P_{\mathbf{HK}}$ is $P_{\mathbf{KH}}$ (see \cite[Remark 10.1]{Janson}), thus if we define
$A: \mathbf{H} \to \mathbf{H}$ by $A=P_{\mathbf{KH}} P_{\mathbf{HK}}$, then $A$ is self-adjoint and
we  have (see, for example, \cite[Appendix H]{Janson}) 
\begin{equation}\label{norms_from_janson}
 \Vert P_{\mathbf{KH}} \Vert=  \Vert P_{\mathbf{HK}} \Vert= \sqrt{\Vert A \Vert} .
\end{equation}

Note that we have
\begin{equation}\label{norm_corr}
 \Vert P_{\mathbf{HK}} \Vert= 
  \sup_{X \in \mathbf{H},  Y \in \mathbf{K}}
 \frac{\langle  P_{\mathbf{K}} (X), Y \rangle}{  \Vert X \Vert  \Vert Y \Vert  } = 
 \sup_{ X \in \mathbf{H}, Y \in \mathbf{K}} 
 \frac{ \langle X,Y \rangle  }{  \Vert X \Vert  \Vert Y \Vert  } .
\end{equation} 
Now $X \in \mathbf{H}$ if and only if $X = \sum_{x \in K_1} \alpha_x \varphi_x$ for some $\alpha \in \R^{K_1}$ and
$Y \in \mathbf{K}$ if and only if $Y = \sum_{y \in K_2} \beta_y \varphi_y$ for some $\beta \in \R^{K_2}$, thus 
we can use this coordinatization and $\mathrm{Cov}[\phi_x\phi_y] = g(x,y)$ to write
\begin{align}
\label{cov_green}
\langle X, Y \rangle= \mathrm{Cov}(X,Y) &= \sum_{x \in K_1, \, y \in K_2} \alpha_x \beta_y g(x,y), \\
 \label{var_green}
  \Vert X \Vert^2 = \mathrm{Var}(X) &=  \sum_{x \in K_1, \, y \in K_1} \alpha_x \alpha_y g(x,y).
 \end{align}
Also note that  $P_{\mathbf{H}}(\varphi_x)\stackrel{\eqref{def_eq_h}}{=}h_x, \, x \in \Z^d$ and that
an analogous formula holds for $P_{\mathbf{K}}(\varphi_x)$. In particular, the entries of the matrices of 
$P_{\mathbf{KH}}$, $P_{\mathbf{HK}}$ and $A=P_{\mathbf{KH}} P_{\mathbf{HK}}$ 
(expressed in the basis $\varphi_x, \, x \in \Z^d$) are all non-negative,
thus we can use the Perron-Frobenius theorem to infer 
that the self-adjoint matrix~$A$ has an eigenvector 
$X_* \in \mathbf{H}$, $\Vert X_* \Vert=1$
such that $\langle X_*, A X_* \rangle = \Vert A \Vert$ and~$X_*$  
has non-negative coordinates 
in the basis $\varphi_x, \, x \in K_1$. We claim that if we define $Y_*=P_{\mathbf{HK}} X_*$, then
the pair $(X_*,Y_*)$ maximizes the correlation functional on the right-hand side of \eqref{norm_corr}:
\begin{gather*}
\Vert Y_* \Vert = \sqrt{ \langle P_{\mathbf{HK}} X_*, P_{\mathbf{HK}} X_*  \rangle   }=
\sqrt{ \langle  X_*, P_{\mathbf{KH}} P_{\mathbf{HK}} X_*  \rangle   }
= \sqrt{\Vert A \Vert} \stackrel{\eqref{norms_from_janson}  }{=}
 \Vert P_{\mathbf{HK}}\Vert, \\
 \frac{ \langle X_*,Y_* \rangle  }{  \Vert X_* \Vert  \Vert Y_* \Vert  } = 
 \frac{ \langle X_*,  P_{\mathbf{HK}} X_* \rangle }{ \Vert P_{\mathbf{HK}}\Vert }=
 \frac{ \langle P_{\mathbf{HK}}  X_*,  P_{\mathbf{HK}} X_* \rangle }{ \Vert P_{\mathbf{HK}}\Vert }=
 \frac{ \Vert A \Vert  }{ \Vert P_{\mathbf{HK}}\Vert } \stackrel{\eqref{norms_from_janson}  }{=}
 \Vert P_{\mathbf{HK}}\Vert.
\end{gather*}
We can thus infer that
the maximum in \eqref{norm_corr} remains unchanged if we assume 
 \begin{equation}\label{sum_equals_1}
\alpha_x \geq 0, \, x \in K_1, \quad \beta_y \geq 0,\, y \in K_2, \quad
  \sum_{x \in K_1} \alpha_x=1 \quad \text{and} \quad
  \sum_{y \in K_2} \beta_y=1.
\end{equation}  
   Using these assumptions we can bound 
   \begin{multline}\label{green_cov_ineqs}
    c'_d  \dist(K_1,K_2)^{2-d} \stackrel{ \eqref{lawler_green} , \eqref{K_1_K_1_far} }{\leq} 
     \min_{x \in K_1,\, y \in K_2} g(x,y)
   \stackrel{ \eqref{cov_green}, \eqref{sum_equals_1} }{\leq} 
   \langle X,Y \rangle \\
   \stackrel{ \eqref{cov_green}, \eqref{sum_equals_1} }{\leq} 
    \max_{x \in K_1,\, y \in K_2} g(x,y) \stackrel{ \eqref{lawler_green}  }{\leq} 
      C'_d  \dist(K_1,K_2)^{2-d}.
   \end{multline}
   A combination of
  \eqref{var_green} and the 
   variational characterization \cite[Lemma~2.3]{JainOrey} of capacity    gives
\begin{equation}\label{variational}
 \sup_{\alpha} \frac{1}{ \Vert X \Vert^2 } = \mathrm{cap}(K_1), \quad
  \sup_{\beta} \frac{1}{ \Vert Y \Vert^2 } = \mathrm{cap}(K_2),
 \end{equation}
where the maximum is taken over all $\alpha \in \R^{K_1}$ and 
$\beta \in \R^{K_1}$ satisfying~\eqref{sum_equals_1}. Putting together  \eqref{norm_corr}, \eqref{green_cov_ineqs} and \eqref{variational}
 we arrive at \eqref{rewritten_cov_norm}. The proof of Proposition~\ref{p_simple} is complete.

\end{proof}

Now we prove the conditional decoupling result:
\begin{proof}[Proof of Theorem~\ref{t_cond}] 
Recall the decomposition  $\phi={\tilde \phi}+h$ from \eqref{decomposition}.
 Let us assume without loss of generality that our
probability space is rich enough to carry an independent copy~$\widehat h$
of the field~$h$. Denote 
\[{\widehat\phi}={\tilde\phi}+{\widehat h}.\]
Clearly, $\phi$ and ${\widehat\phi}$ have the same law. 
Let ${\widehat G}_\delta$ be the event defined as in~\eqref{def_G},
but with ${\widehat h}$ replacing~$h$.
Then, write
\begin{align}
 \IE\big(f_2(\phi)\mid \phi_{K_1}\big)\I{G_\delta}
   &= \IE\big(f_2({\tilde\phi}+h)\mid \phi_{K_1}\big)\I{G_\delta}\nonumber\\
   &=\IE\big(f_2({\widehat\phi}+h-{\widehat h})\mid
 \phi_{K_1}\big)\I{G_\delta}\nonumber\\
   &= \IE\big(f_2({\widehat\phi}+h-{\widehat h})\I{G_\delta\cap {\widehat G}_\delta}
  \mid \phi_{K_1}\big)\nonumber\\
& ~~~~~~~ +  \IE\big(f_2({\widehat\phi}+h-{\widehat h})
\I{G_\delta\cap {\widehat G}_\delta^c}
  \mid \phi_{K_1}\big)\nonumber\\
  &=: T_1 + T_2.
\label{cond_calculation}
\end{align}
Clearly, we have
\begin{equation}
\label{boundII}
 0\leq T_2 \leq \IP[{\widehat G}_\delta^c]\I{G_\delta}.
\end{equation}
Since, by construction, $\widehat \phi$ is independent of~$\phi_{K_1}$, moreover
$|h-{\widehat h}|\leq \delta$ on $G_\delta\cap {\widehat G}_\delta$ 
and~$f_2$ is increasing, we can write
\begin{align}
 T_1 &\leq \IE\big(f_2({\widehat\phi}+\delta)\I{G_\delta\cap {\widehat G}_\delta}
  \mid \phi_{K_1}\big)\nonumber\\
  &\leq \IE\big(f_2({\widehat\phi}+\delta)\mid \phi_{K_1}\big)
\I{G_\delta}\nonumber\\
  &= \IE\big(f_2({\widehat\phi}+\delta)\big)\I{G_\delta}.
\label{upperboundI}
\end{align}
Also, we have
\begin{align}
 T_1 &\geq \IE\big(f_2({\widehat\phi}-\delta)\I{{\widehat G}_\delta}
  \mid \phi_{K_1}\big)\I{G_\delta}\nonumber\\
  &=\IE\big(f_2({\widehat\phi}-\delta)(1-\I{{\widehat G}_\delta^c})
  \big)\I{G_\delta}\nonumber\\
  &\geq \IE\big(f_2({\widehat\phi}-\delta)\big)\I{G_\delta} 
  - \IP[{\widehat G}_\delta^c]\I{G_\delta}.
\label{lowerboundI}
\end{align}
Inserting \eqref{boundII}--\eqref{lowerboundI} into~\eqref{cond_calculation}
and using the fact that $\widehat \phi$ and~$\phi$ are equally distributed and 
$\IP[{\widehat G}_\delta^c]=\IP[G_\delta^c]$, we conclude the proof 
of Theorem~\ref{t_cond}.
\end{proof}

\begin{proof}[Proof of Corollary~\ref{t_decouple}]
 Now, let~$f_1:\R^{\Z^d}\rightarrow [0,1]$ be a function supported on~$K_1$.
Since
\[
 \IE f_1(\phi) - \IP[G_\delta^c] \leq \IE\big( f_1(\phi)\I{G_\delta}\big)
 \leq  \IE f_1(\phi),
\]
it is then straightforward to obtain~\eqref{uncond_incr} by
 multiplying~\eqref{cond_incr} by~$f_1(\phi)$ and integrating.
\end{proof}

\begin{proof}[Proof of Proposition~\ref{p_ocG}]

Define the events
 \[
  \Lambda_{\delta,x} = \{|h_x|\leq \delta/2\}, \qquad x\in H_1^{(=s)} .
 \]
Clearly, $h_x$ is a centered Gaussian random variable, thus we can use \eqref{def_eq_h} and
$\mathbb E [\phi_x\phi_y] = g(x,y)$ to
 bound the variance of $h_x$, $x\in H_1^{(=s)}$:
\begin{align*}
 \Var h_x &= \sum_{y\in K_1} P_x [ H_{ K_1} < \infty, \, 
 X_{H_{K_1}} =y ]
   \sum_{z\in  K_1} P_x [ H_{ K_1} < \infty, \, 
 X_{H_{K_1}} =z ] g(z,y)\\
   & \stackrel{(*)}{=} \sum_{y\in  K_1} P_x [ H_{ K_1} < \infty, \, 
 X_{H_{K_1}} =y ] g(x,y)
   \leq \sup_{y\in  K_1} g(x,y)
   \stackrel{ \eqref{auxialiary_set_properties} }{\leq} g_s,
\end{align*}
where $(*)$ holds by  the strong Markov property of simple random walk:
\begin{multline*}
 g(x,y)= E_x \left[\, \sum_{n=0}^{\infty} \I{[X_n=y]} \, \right] =
 E_x \left[\, \sum_{n=H_{K_1}}^{\infty} \I{[X_n=y]} \, \right]=
 \\
  E_x \left( g(X_{H_{K_1}},y); \;H_{ K_1} < \infty \right)=
  \sum_{z\in  K_1} P_x [ H_{ K_1} < \infty, \, 
 X_{H_{K_1}} =z ] g(z,y).
 \end{multline*}

Thus we can use the exponential Chebyshev's inequality to bound
\[
 \IP[\Lambda_{\delta,x}^c] 
 \leq 
 2\exp\Big(-\frac{\delta^2}{8g_s} \Big).
\]
Observe that by \eqref{auxialiary_set_properties} any nearest-neighbor walk from $K_2$ to $K_1$ must pass through $H_1^{(= s)}$, so
 by the strong Markov property of the simple random walk on $\Z^d$ and \eqref{def_eq_h},
 for any $y\in K_2$, the value of $h_y$ is a weighted sum of the values
$\big(h_x, x\in H_1^{(=s)}\big)$, with total weight at most~$1$. In particular, we have
\[G_\delta^c \subseteq \bigcup_{x\in H_1^{(=s)}} \Lambda_{\delta,x}^c.\]
 Using the union bound we conclude
the proof of Proposition~\ref{p_ocG}. 
\end{proof}

\section*{Acknowledgements} The work of Serguei Popov was partially
supported by CNPq (300328/2005--2) and FAPESP (2009/52379--8).
The work of Bal\'azs R\'ath is partially supported by OTKA (Hungarian
National Research Fund) grant K100473, the Postdoctoral Fellowship of
the Hungarian Academy of Sciences and the Bolyai Research Scholarship
of the Hungarian Academy of Sciences.

The authors also thank the organizers of the conference 
\textit{Random Walks: Crossroads
and Perspectives} (Budapest, June 24--28, 2013),
for providing the opportunity for the authors to meet 
and work on this topic. 
This paper was written while B.R.\ was a postdoctoral fellow of the University of British Columbia.

We thank Prof.\ Alain-Sol Sznitman for pointing out the reference \cite{Janson} to us, moreover
 Pierre-Fran\c cois Rodriguez, Alexander Drewitz and a very thorough anonymous referee for 
reading and commenting on the manuscript.

\end{document}